\documentclass{article}%
\usepackage{amsfonts}
\usepackage{amsmath}
\usepackage{amssymb}
\usepackage{graphicx}%
\providecommand{\U}[1]{\protect\rule{.1in}{.1in}}
\newtheorem{theorem}{Theorem}

\newtheorem{definition}[theorem]{Definition}

\newtheorem{lemma}[theorem]{Lemma}

\newenvironment{proof}[1][Proof]{\noindent\textbf{#1.} }{\ \rule{0.5em}{0.5em}}
\begin{document}

\title{Approximation by Genuine $q$-Bernstein-Durrmeyer Polynomials in Compact Disks
in the case $q>1$}
\author{Nazim I. Mahmudov}
\date{}
\maketitle

\begin{abstract}
This paper deals with approximating properties of the newly defined
$q$-generalization of the genuine Bernstein-Durrmeyer polynomials in the case
$q>1$, whcih are no longer positive linear operators on $C[0,1]$. Quantitative
estimates of the convergence, the Voronovskaja type theorem and saturation of
convergence for complex genuine $q$-Bernstein-Durrmeyer polynomials attached
to analytic functions in compact disks are given. In particular, it is proved
that for functions analytic in $\left\{  z\in\mathbb{C}:\left\vert
z\right\vert <R\right\}  $, $R>q,$ the rate of approximation by the genuine
$q$-Bernstein-Durrmeyer polynomials ($q>1$) is of order $q^{-n}$ versus $1/n$
for the classical genuine Bernstein-Durrmeyer polynomials. We give explicit
formulas of Voronovskaja type for the genuine $q$-Bernstein-Durrmeyer for
$q>1$.

\end{abstract}

\section{Introduction}

In several recent papers, convergence properties of complex $q$-Bernstein
polynomials, proposed by Phillips \cite{phil}, attached to an analytic
function $f$ in closed disks, were intensively studied. Ostrovska \cite{ost1},
\cite{ost2}, and Wang and Wu \cite{wang}, \cite{wu} have investigated
convergence properies of $B_{n,q}$ in the case $q>1.$ In the case $q>1$, the
$q$-Bernstein polynomials are no longer positive operators, however, for a
function analytic in a disc $\mathbb{D}_{R}:=\left\{  z\in\mathbb{C}%
:\left\vert z\right\vert <R\right\}  ,\ R>q$, it was proved in \cite{ost1}
that the rate of convergence of $\left\{  B_{n,q}\left(  f;z\right)  \right\}
$ to $f\left(  z\right)  $ has the order $q^{-n}$ (versus $1/n$ for the
classical Bernstein polynomials). Moreover, Ostrovska \cite{ost2} obtained
Voronovskaya type theorem for monomials. If $q\geq1$ then qualitative
Voronovskaja-type and saturation results for complex $q$-Bernstein polynomials
were obtained in Wang-Wu \cite{wang}. Wu \cite{wu} studied saturation of
convergence on the interval $[0,1]$ for the $q$-Bernstein polynomials of a
continuous function $f$ for arbitrary fixed $q>1$.

Genuine Bernstein--Durrmeyer operators were first considered by Chen \cite{28}
and Goodman and Sharma \cite{74} around 1987. In recent years, the genuine
Bernstein--Durrmeyer operators have been investigated intensively by a number
of authors. Among the many articles written on the genuine
Bernstein--Durrmeyer operators, we mention here only the ones by Gonska and
etc \cite{gonska}, by Parvanov and Popov \cite{pp}, by Sauer \cite{sauer}, by
Waldron \cite{waldron}, and the book of P\'{a}lt\'{a}nea \cite{pal}.

On the other hand, Gal \cite{gal2} obtained quantitative estimates of the
convergence and of the Voronovskaja theorem in compact disks, for the complex
genuine Bernstein--Durrmeyer polynomials attached to analytic functions.
Besides, in other very recent papers, similar studies were done for complex
Bernstein-Durrmeyer operators in Anastassiou-Gal \cite{AG}, for complex
Bernstein-Durrmeyer operators based on Jacobi weights in Gal \cite{SG4}, for
complex genuine Bernstein-Durrmeyer operator in Gal \cite{SG3}, for complex
$q$-genuine Bernstein-Durrmeyer operator in Mahmudov \cite{mah3} and for other
kinds of complex Durrmeyer operators in Mahmudov \cite{mah4} and
Gal-Gupta-Mahmudov \cite{GGY}. Also, for the case $q>1,$ exact quantitative
estimates and quantitative Voronovskaja-type results for complex $q$-Lorentz
polynomials, $q$-Stancu polynomials, $q$-Stancu-Faber polynomials,
$q$-Bernstein-Faber polynomials, $q$-Kantorovich polynomials, $q$%
-Sz\'{a}sz-Mirakjan operators obtained by different researchers are collected
in the recent book of Gal \cite{galbook}.

In this paper we define the genuine $q$-Bernstein-Durrmeyer polynomials for
$q>1.$ Note that similar to the $q$-Bernstein operators the genuine
$q$-Bernstein-Durrmeyer operators in the case $q>1$ are not positive operators
on $C\left[  0,1\right]  $. The lack of positivity makes the investigation of
convergence in the case $q>1$ essentially more difficult than that for
$0<q<1$. We present upper estimates in approximation and we prove the
Voronovskaja type convergence theorem in compact disks in $\mathbb{C}$,
centered at origin, with quantitative estimate of this convergence. These
results allow us to obtain the exact degrees of approximation by complex
genuine $q$-Bernstein-Durrmeyer polynomials. Our results show that
approximation properties of the complex genuine $q$-Bernstein-Durrmeyer
polynomials are better than approximation properties of the complex
Bernstein-Durrmeyer polynomials considered in \cite{gal2}.

\section{Formulation}

Let $q>0.$ For any $n\in\mathbb{N}\cup\left\{  0\right\}  $, the $q$-integer
$\left[  n\right]  _{q}$ is defined by%
\[
\left[  n\right]  _{q}:=1+q+...+q^{n-1},\ \ \ \left[  0\right]  _{q}:=0;
\]
and the $q$-factorial $\left[  n\right]  _{q}!$ by%
\[
\left[  n\right]  _{q}!:=\left[  1\right]  _{q}\left[  2\right]
_{q}...\left[  n\right]  _{q},\ \ \ \ \left[  0\right]  _{q}!:=1.
\]
For integers $0\leq k\leq n$, the $q$-binomial is defined by%
\[
\left[
\begin{array}
[c]{c}%
n\\
k
\end{array}
\right]  _{q}:=\frac{\left[  n\right]  _{q}!}{\left[  k\right]  _{q}!\left[
n-k\right]  _{q}!}.
\]
For $q=1$ we obviously get $\left[  n\right]  _{q}=n$, $\left[  n\right]
_{q}!=n!$, $\left[
\begin{array}
[c]{c}%
n\\
k
\end{array}
\right]  _{q}=\left(
\begin{array}
[c]{c}%
n\\
k
\end{array}
\right)  .$ Moreover%
\[
\left(  1-z\right)  _{q}^{n}:=%
%TCIMACRO{\dprod _{s=0}^{n-1}}%
%BeginExpansion
{\displaystyle\prod_{s=0}^{n-1}}
%EndExpansion
\left(  1-q^{s}z\right)  ,\ \ p_{n,k}\left(  q;z\right)  :=\left[
\begin{array}
[c]{c}%
n\\
k
\end{array}
\right]  _{q}z^{k}\left(  1-z\right)  _{q}^{n-k},\ \ z\in\mathbb{C}.
\]

For fixed $q>0,\ q\neq1$, we denote the $q$-derivative $D_{q}f\left(
z\right)  $ of $f$ by%
\[
D_{q}f\left(  z\right)  =\left\{
\begin{tabular}
[c]{lll}%
$\frac{f\left(  qz\right)  -f\left(  z\right)  }{\left(  q-1\right)  z},$ &
$z\neq0,$ & \\
&  & \\
$f^{\prime}\left(  0\right)  ,$ & $z=0.$ &
\end{tabular}
\ \ \ \right.
\]

The $q$-analogue of integration in the interval $[0,A]$ (see \cite{AAR}) is
defined by%
\[
\int_{0}^{A}f\left(  t\right)  d_{q}t:=A\left(  1-q\right)  \sum_{n=0}%
^{\infty}f\left(  Aq^{n}\right)  q^{n},\ \ \ 0<q<1.
\]
Let $\mathbb{D}_{R}$ be a disc $\mathbb{D}_{R}:=\left\{  z\in\mathbb{C}%
:\left\vert z\right\vert <R\right\}  $ in the complex plane $\mathbb{C}$.
Denote by $H\left(  \mathbb{D}_{R}\right)  $ the space of all analytic
functions on $\mathbb{D}_{R}$. For $f\in H\left(  \mathbb{D}_{R}\right)  $ we
assume that $f\left(  z\right)  =\sum_{m=0}^{\infty}a_{m}z^{m}$.

\begin{definition}
\label{def1}For $f:\left[  0,1\right]  \rightarrow\mathbb{C}$, the genuine
$q$-Bernstein-Durrmeyer operator is defined as follows:
\begin{equation}
U_{n,q}\left(  f;z\right)  :=\left\{
\begin{tabular}
[c]{ll}%
$f\left(  0\right)  p_{n,0}\left(  q;z\right)  +f\left(  1\right)
p_{n,n}\left(  q;z\right)  $ & \\
$+\left[  n-1\right]  _{q}%
%TCIMACRO{\dsum _{k=1}^{n-1}}%
%BeginExpansion
{\displaystyle\sum_{k=1}^{n-1}}
%EndExpansion
q^{1-k}p_{n,k}\left(  q;z\right)  \int_{0}^{1}p_{n-2,k-1}\left(  q;qt\right)
f\left(  t\right)  d_{q}t,$ & $0<q<1,$\\
$f\left(  0\right)  p_{n,0}\left(  z\right)  +f\left(  1\right)
p_{n,n}\left(  z\right)  +\left(  n-1\right)
%TCIMACRO{\dsum _{k=1}^{n-1}}%
%BeginExpansion
{\displaystyle\sum_{k=1}^{n-1}}
%EndExpansion
p_{n,k}\left(  z\right)  \int_{0}^{1}p_{n-2,k-1}\left(  t\right)  f\left(
t\right)  dt,$ & $q=1,$\\
$f\left(  0\right)  p_{n,0}\left(  q;z\right)  +f\left(  1\right)
p_{n,n}\left(  q;z\right)  $ & \\
$+\left[  n-1\right]  _{q^{-1}}%
%TCIMACRO{\dsum _{k=1}^{n-1}}%
%BeginExpansion
{\displaystyle\sum_{k=1}^{n-1}}
%EndExpansion
q^{k-1}p_{n,k}\left(  q;z\right)  \int_{0}^{1}p_{n-2,k-1}\left(  q^{-1}%
;q^{-1}t\right)  f\left(  q^{k-n}t\right)  d_{q^{-1}}t,$ & $q>1,$%
\end{tabular}
\right.  \label{bd}%
\end{equation}
where for $n=1$ the sum is empty, i.e., equal to $0.$
\end{definition}

$U_{n,q}\left(  f;z\right)  $ are linear operators reproducing linear
functions and interpolating every function $f\in C\left[  0,1\right]  $ at $0$
and $1$. The genuine $q$-Bernstein-Durrmeyer operators are positive operators
on $C\left[  0,1\right]  $ for $0<q\leq1,$ and they are not positive for
$q>1.$ As a consequence, the cases $0<q\leq1$ and $q>1$ are not similar to
each other regarding the convergence. For $q\rightarrow1^{-}$ and
$q\rightarrow1^{+}$ we recapture the classical ($q=1$) genuine
Bernstein-Durrmeyer polynomials.

We start with the following quantitative estimates of the convergence for
complex $q$-Bernstein-Durrmeyer polynomials attached to an analytic function
in a disk of radius $R>1$ and center $0$.

\begin{theorem}
\label{t:convergence}Let $f\in H\left(  \mathbb{D}_{R}\right)  $, $1\leq
r<\dfrac{R}{q}$ and $q>1$. Then for all $\left\vert z\right\vert \leq r$ we
have
\[
\left\vert U_{n,q}\left(  f;z\right)  -f\left(  z\right)  \right\vert
\leq\frac{r\left(  1+r\right)  }{\left[  n+1\right]  _{q}}%
%TCIMACRO{\dsum \limits_{m=2}^{\infty}}%
%BeginExpansion
{\displaystyle\sum\limits_{m=2}^{\infty}}
%EndExpansion
\left\vert a_{m}\right\vert m\left(  m-1\right)  q^{m-2}r^{m-2}.
\]

\end{theorem}

Theorem \ref{t:convergence} says that for functions analytic in $\mathbb{D}%
_{R}$, $R>q,$ the rate of approximation by the genuine $q$-Bernstein-Durrmeyer
polynomials ($q>1$) is of order $q^{-n}$ versus $1/n$ for the classical
genuine Bernstein-Durrmeyer polynomials, see \cite{gal2}.

The Voronovskaja theorem for the real case with a quantitative estimate is
obtained by Gonska, Pi\c{t}ul and Ra\c{s}a \cite{gonska2} in the following
form:%
\[
\left\vert U_{n}\left(  f;x\right)  -f\left(  x\right)  -\frac{x\left(
1-x\right)  }{n+1}f^{\prime\prime}\left(  z\right)  \right\vert \leq
\frac{x\left(  1-x\right)  }{n+1}\omega\left(  f^{\prime\prime}\frac{2}%
{3\sqrt{n+3}}\right)  ,
\]
for all $n\in\mathbb{N},\ 0\leq x\leq1.$ For the complex genuine
$q$-Bernstein-Durrmeyer ( $0<q\leq1$) a quantitative estimate is obtained by
Gal \cite{gal2}( $q=1$) and Mahmudov \cite{mah3}($0<q<1$) in the following
form:%
\[
\left\vert U_{n,q}\left(  f;z\right)  -f\left(  z\right)  -\frac{z\left(
1-z\right)  }{\left[  n+1\right]  _{q}}f^{\prime\prime}\left(  z\right)
\right\vert \leq\frac{M_{r,f}}{\left[  n\right]  _{q}^{2}},\ \ 0<q\leq1,
\]
for all $n\in\mathbb{N},\ \left\vert z\right\vert \leq r$.

To formulate and prove the Voronovskaja type theorem with a quantitative
estimate in the case $q>1$ we introduce a function $L_{q}\left(  f;z\right)  $.

Let $R>q\geq1$ and let $f\in H\left(  \mathbb{D}_{R}\right)  $. For
$\left\vert z\right\vert <R/q^{2}$, we define%
\begin{equation}
L_{q}\left(  f;z\right)  :=\frac{\left(  1-z\right)  q\left(  D_{q}f\left(
z\right)  -D_{q^{-1}}f\left(  z\right)  \right)  }{q-1}\ \ \ \ \ \text{for
}q>1 \label{lq}%
\end{equation}
and for $0<q\leq1$,%
\[
L_{q}\left(  f;z\right)  =L_{1}\left(  f;z\right)  :=f^{\prime\prime}\left(
z\right)  z\left(  1-z\right)  .
\]

The next theorem gives Voronovskaja type result in compact disks, for complex
$q$-Bernstein-Durrmeyer polynomials attached to an analytic function in
$\mathbb{D}_{R}$, $R>q^{2}>1$ and center $0$ in terms of the function
$L_{q}\left(  f;z\right)  $.

\begin{theorem}
\label{t:voronovskaja}Let $f\in H\left(  \mathbb{D}_{R}\right)  $, $1\leq
r<\dfrac{R}{q^{2}}$ and $q>1$. The following Voronovskaja-type result holds%
\[
\left\vert U_{n,q}\left(  f;z\right)  -f\left(  z\right)  -\frac{1}{\left[
n+1\right]  _{q}}L_{q}\left(  f;z\right)  \right\vert \leq\frac{4r^{2}\left(
1+r\right)  ^{2}}{\left[  n+1\right]  _{q}^{2}}%
%TCIMACRO{\dsum \limits_{m=3}^{\infty}}%
%BeginExpansion
{\displaystyle\sum\limits_{m=3}^{\infty}}
%EndExpansion
\left\vert a_{m}\right\vert \left(  m-1\right)  ^{2}\left(  m-2\right)
^{2}\left(  q^{2}r\right)  ^{m-2}.
\]
for all $n\in\mathbb{N}$, $\left\vert z\right\vert \leq r.$
\end{theorem}

Now we are in position to prove that the order of approximation in Theorem
\ref{t:convergence} is exactly $q^{-n}$ versus $1/n$ for the classical genuine
Bernstein-Durrmeyer polynomials, see \cite{gal2}.

\begin{theorem}
\label{t:exactdegree}Let $1<q<R$, $1\leq r<\frac{R}{q^{2}}$ and $f\in H\left(
\mathbb{D}_{R}\right)  $. If $f$ is not a polynomial of degree $\leq1$, the
estimate%
\[
\left\Vert U_{n,q}\left(  f\right)  -f\right\Vert _{r}\geq\frac{1}{\left[
n+1\right]  _{q}}C_{r,q}\left(  f\right)  ,\ \ \ n\in\mathbb{N},
\]
holds, where the constant $C_{r,q}\left(  f\right)  $ depends on $f,$ $q$ and
$r$ but is independent of $n$.
\end{theorem}

From Theorem \ref{t:voronovskaja} we conclude that for $q>1,$ $\left[
n+1\right]  _{q}\left(  U_{n,q}\left(  f;z\right)  -f\left(  z\right)
\right)  \rightarrow L_{q}\left(  f;z\right)  $ in $H\left(  \mathbb{D}%
_{R/q}\right)  $ and therefore, $L_{q}\left(  f;z\right)  \in H\left(
\mathbb{D}_{R/q}\right)  $. Furthermore, we have the following saturation of
convergence for the genuine $q$-Bernstein-Durrmeyer polynomials for fixed
$q>1$.

\begin{theorem}
\label{t:saturation}Let $1<q<R$, $1\leq r<\frac{R}{q^{2}}$. If a function $f$
is analytic in the disc $\mathbb{D}_{R/q}$, then $\left\vert U_{n,q}\left(
f;z\right)  -f\left(  z\right)  \right\vert =o\left(  q^{-n}\right)  $ for
infinite number of points having an accumulation point on $\mathbb{D}_{R/q}$
if and only if $f$ is linear.
\end{theorem}

The next theorem shows that $L_{q}\left(  f;z\right)  ,$ $q\geq1,$ is
continuous in the parameter $q$ for $f\in H\left(  \mathbb{D}_{R}\right)  $,
$R>1$.

\begin{theorem}
\label{t:continuous}Let $R>1$ and $f\in H\left(  \mathbb{D}_{R}\right)  $.
Then for any $r,$ $0<r<R$,%
\[
\lim_{q\rightarrow1+}L_{q}\left(  f;z\right)  =L_{1}\left(  f;z\right)
\]
uniformly on $\mathbb{D}_{R}$.
\end{theorem}

\section{Auxiliary results}

The $q$-analogue of Beta function for $0<q<1$ (see \cite{AAR}) is defined as
\[
B_{q}(m,n)=\int_{0}^{1}t^{m-1}(1-qt)_{q}^{n-1}d_{q}t,\,\,\,m,n>0,\ \ 0<q<1.
\]
Since we consider the case $q>1$, we need to use $B_{q^{-1}}(m,n):$%
\[
B_{q^{-1}}(m,n)=\int_{0}^{1}t^{m-1}(1-q^{-1}t)_{q^{-1}}^{n-1}d_{q^{-1}%
}t,\,\,\,m,n>0,\ \ 0<q^{-1}<1.
\]
Also, it is known that
\[
B_{q^{-1}}(m,n)=\frac{[m-1]_{q^{-1}}![n-1]_{q^{-1}}!}{[m+n-1]_{q^{-1}}%
!},\ \ 0<q^{-1}<1.
\]
For $m=0,1,...$, we have%
\begin{align*}
&  \left[  n-1\right]  _{q^{-1}}q^{k-1}\int_{0}^{1}t^{m}p_{n-2,k-1}\left(
q^{-1};q^{-1}t\right)  d_{q^{-1}}t\\
&  =\left[  n-1\right]  _{q^{-1}}\left[
\begin{array}
[c]{c}%
n-2\\
k-1
\end{array}
\right]  _{q^{-1}}q^{m\left(  k-n\right)  }\int_{0}^{1}t^{k+m-1}\left(
1-q^{-1}t\right)  _{q^{-1}}^{n-k-1}d_{q^{-1}}t\\
&  =q^{m\left(  k-n\right)  }\frac{\left[  n-1\right]  _{q^{-1}}!}{\left[
k-1\right]  _{q^{-1}}!\left[  n-k-1\right]  _{q^{-1}}!}B_{q^{-1}}(k+m,n-k)\\
&  =q^{m\left(  k-n\right)  }\frac{\left[  n-1\right]  _{q^{-1}}!}{\left[
k-1\right]  _{q^{-1}}!\left[  n-k-1\right]  _{q^{-1}}!}\frac{\left[
k+m-1\right]  _{q^{-1}}!\left[  n-k-1\right]  _{q^{-1}}!}{\left[
k+m+n-k-1\right]  _{q^{-1}}!}\\
&  =\frac{\left[  n-1\right]  _{q}!\left[  k+m-1\right]  _{q}!}{\left[
k-1\right]  _{q}!\left[  n+m-1\right]  _{q}!}=\frac{\left[  k+m-1\right]
_{q}...\left[  k\right]  _{q}}{\left[  n+m-1\right]  _{q}...\left[  n\right]
_{q}}.
\end{align*}
Thus, we get the following formula for $U_{n,q}\left(  e_{m};z\right)  :$
\begin{align}
U_{n,q}\left(  e_{m};z\right)   &  =f\left(  0\right)  p_{n,0}\left(
q;z\right)  +f\left(  1\right)  p_{n,n}\left(  q;z\right)  \nonumber\\
&  +\left[  n-1\right]  _{q^{-1}}%
%TCIMACRO{\dsum _{k=1}^{n-1}}%
%BeginExpansion
{\displaystyle\sum_{k=1}^{n-1}}
%EndExpansion
p_{n,k}\left(  q;z\right)  \int_{0}^{1}p_{n-2,k-1}\left(  q^{-1}%
;q^{-1}t\right)  f\left(  q^{k-n}t\right)  d_{q^{-1}}t\nonumber\\
&  =z^{n}+%
%TCIMACRO{\dsum _{k=1}^{n-1}}%
%BeginExpansion
{\displaystyle\sum_{k=1}^{n-1}}
%EndExpansion
p_{n,k}\left(  q;z\right)  \frac{\left[  k+m-1\right]  _{q}...\left[
k\right]  _{q}}{\left[  n+m-1\right]  _{q}...\left[  n\right]  _{q}%
}.\label{u1}%
\end{align}
Note for $m=0,1,2$ we have%
\[
U_{n,q}\left(  e_{0};z\right)  =1,\ \ \ U_{n,q}\left(  e_{1};z\right)
=z,\ \ \ U_{n,q}\left(  e_{2};z\right)  =z^{2}+\frac{\left(  1+q\right)
z\left(  1-z\right)  }{\left[  n+1\right]  }.
\]

\begin{lemma}
$U_{n,q}\left(  e_{m};z\right)  $ is a polynomial of degree less than or equal
to $\min\left(  m,n\right)  $ and
\[
U_{n,q}\left(  e_{m};z\right)  =\frac{\left[  n-1\right]  _{q}!}{\left[
n+m-1\right]  _{q}!}\sum_{s=1}^{m}S_{q}\left(  m,s\right)  \left[  n\right]
_{q}^{s}B_{n,q}\left(  e_{s};z\right)  .
\]

\end{lemma}

\begin{proof}
From (\ref{u1}) it follows that%
\begin{align*}
U_{n,q}\left(  e_{m};z\right)   &  =%
%TCIMACRO{\dsum _{k=1}^{n}}%
%BeginExpansion
{\displaystyle\sum_{k=1}^{n}}
%EndExpansion
p_{n,k}\left(  q;z\right)  \frac{\left[  k+m-1\right]  _{q}...\left[
k\right]  _{q}}{\left[  n+m-1\right]  _{q}...\left[  n\right]  _{q}}\\
&  =\frac{\left[  n-1\right]  _{q}!}{\left[  n+m-1\right]  _{q}!}%
%TCIMACRO{\dsum _{k=1}^{n}}%
%BeginExpansion
{\displaystyle\sum_{k=1}^{n}}
%EndExpansion
\left[  k\right]  _{q}\left[  k+1\right]  _{q}...\left[  k+m-1\right]
_{q}p_{n,k}(q;z).
\end{align*}
Now using%
\begin{equation}
\left[  k\right]  _{q}\left[  k+1\right]  _{q}...\left[  k+m-1\right]  _{q}=%
%TCIMACRO{\dprod \limits_{s=0}^{m-1}}%
%BeginExpansion
{\displaystyle\prod\limits_{s=0}^{m-1}}
%EndExpansion
\left(  q^{s}\left[  k\right]  _{q}+\left[  s\right]  _{q}\right)  =\sum
_{s=1}^{m}S_{q}\left(  m,s\right)  \left[  k\right]  _{q}^{s},
\label{stirling}%
\end{equation}
where $S_{q}\left(  m,s\right)  >0$, $s=1,2,...,m$, are the constants
independent of $k,$ we get
\begin{align*}
U_{n,q}\left(  e_{m};z\right)   &  =\frac{\left[  n-1\right]  _{q}!}{\left[
n+m-1\right]  _{q}!}%
%TCIMACRO{\dsum _{k=0}^{n}}%
%BeginExpansion
{\displaystyle\sum_{k=0}^{n}}
%EndExpansion
\sum_{s=1}^{m}S_{q}\left(  m,s\right)  \left[  k\right]  _{q}^{s}%
p_{n,k}(q;z)\\
&  =\frac{\left[  n-1\right]  _{q}!}{\left[  n+m-1\right]  _{q}!}\sum
_{s=1}^{m}S_{q}\left(  m,s\right)  \left[  n\right]  _{q}^{s}B_{n,q}\left(
e_{s};z\right)  ,
\end{align*}
Since $B_{n,q}(e_{s};z)$ is a polynomial of degree less than or equal to
$\min\left(  s,n\right)  $ and $S_{q}\left(  m,s\right)  >0$, $s=1,2,...,m$,
it follows that $U_{n,q}\left(  e_{m};z\right)  $ is a polynomial of degree
less than or equal to $\min\left(  m,n\right)  $.
\end{proof}

\begin{lemma}
The numbers $S_{q}\left(  m,s\right)  ,\ \left(  m,s\right)  \in\left(
\mathbb{N}\cup\left\{  0\right\}  \right)  \times\left(  \mathbb{N}%
\cup\left\{  0\right\}  \right)  ,$ given by (\ref{stirling}) enjoy the
following properties
\begin{align*}
S_{q}\left(  0,0\right)   &  =1,\ \ S_{q}\left(  m,0\right)  =0,\ \ m\in N,\\
S_{q}\left(  m+1,s\right)   &  =\left[  m\right]  _{q}S_{q}\left(  m,s\right)
+q^{m}S_{q}\left(  m,s-1\right)  ,\ \ \ m\in N_{0},\ s\in N,\\
S_{q}\left(  m+1,m+1\right)   &  =q^{m}S_{q}\left(  m,m\right)  ,\ \ \ S_{q}%
\left(  m,s\right)  =0\ \ \text{for\ }\ s>m.
\end{align*}

\end{lemma}

Also, the following lemma holds.

\begin{lemma}
\label{l:id}For all $m,n\in\mathbb{N}$ the identity%
\[
\frac{\left[  n-1\right]  _{q}!}{\left[  n+m-1\right]  _{q}!}\sum_{s=1}%
^{m}S_{q}\left(  m,s\right)  \left[  n\right]  _{q}^{s}=1,
\]
holds.
\end{lemma}

\begin{proof}
It follows from end points interpolation property of $U_{n,q}\left(
e_{m};z\right)  $ and $B_{n,q}\left(  e_{s};z\right)  .$Indeed%
\[
1=U_{n,q}\left(  e_{m};1\right)  =\frac{\left[  n-1\right]  _{q}!}{\left[
n+m-1\right]  _{q}!}\sum_{s=1}^{m}S_{q}\left(  m,s\right)  \left[  n\right]
_{q}^{s}B_{n,q}\left(  e_{s};1\right)  =\frac{\left[  n-1\right]  _{q}%
!}{\left[  n+m-1\right]  _{q}!}\sum_{s=1}^{m}S_{q}\left(  m,s\right)  \left[
n\right]  _{q}^{s}.
\]

\end{proof}

Lemma \ref{l:id} implies that for all $m,n\in\mathbb{N}$ and $\left\vert
z\right\vert \leq r$ we have%
\begin{align}
\left\vert U_{n,q}\left(  e_{m};z\right)  \right\vert  &  \leq\frac{\left[
n-1\right]  _{q}!}{\left[  n+m-1\right]  _{q}!}\sum_{s=1}^{m}S_{q}\left(
m,s\right)  \left[  n\right]  _{q}^{s}\left\vert B_{n,q}\left(  e_{s}%
;z\right)  \right\vert \nonumber\\
&  \leq\frac{\left[  n-1\right]  _{q}!}{\left[  n+m-1\right]  _{q}!}\sum
_{s=1}^{m}S_{q}\left(  m,s\right)  \left[  n\right]  _{q}^{s}r^{s}\leq r^{m}.
\label{inq1}%
\end{align}

For our purpose first we need a recurrence formula for $U_{n,q}\left(
e_{m};z\right)  .$

\begin{lemma}
\label{Lem:rec}For all $m,n\in\mathbb{N}\cup\left\{  0\right\}  $ and
$z\in\mathbb{C}$ we have%
\begin{equation}
U_{n,q}\left(  e_{m+1};z\right)  =\frac{q^{m}z\left(  1-z\right)  }{\left[
n+m\right]  _{q}}D_{q}U_{n,q}\left(  e_{m};z\right)  +\frac{q^{m}\left[
n\right]  z+\left[  m\right]  _{q}}{\left[  n+m\right]  _{q}}U_{n,q}\left(
e_{m};z\right)  . \label{rec1}%
\end{equation}

\end{lemma}

\begin{proof}
By simple calculation we obtain (see \cite{mah5})%
\[
z\left(  1-z\right)  D_{q}\left(  p_{n,k}\left(  q;z\right)  \right)  =\left(
\left[  k\right]  _{q}-\left[  n\right]  _{q}z\right)  p_{n,k}\left(
q;z\right)  ,
\]
and%
\begin{align*}
U_{n,q}\left(  e_{m};z\right)   &  =z^{n}+%
%TCIMACRO{\dsum _{k=1}^{n-1}}%
%BeginExpansion
{\displaystyle\sum_{k=1}^{n-1}}
%EndExpansion
p_{n,k}\left(  q;z\right)  \frac{\left[  k+m-1\right]  _{q}...\left[
k\right]  _{q}}{\left[  n+m-1\right]  _{q}...\left[  n\right]  _{q}}\\
&  =z^{n}+%
%TCIMACRO{\dsum _{k=1}^{n-1}}%
%BeginExpansion
{\displaystyle\sum_{k=1}^{n-1}}
%EndExpansion
p_{n,k}\left(  q;z\right)  I_{k,m},\\
I_{k,m}  &  :=\frac{\left[  k+m-1\right]  _{q}...\left[  k\right]  _{q}%
}{\left[  n+m-1\right]  _{q}...\left[  n\right]  _{q}}.
\end{align*}
It follows that%
\begin{align}
&  z\left(  1-z\right)  D_{q}U_{n,q}\left(  e_{m};z\right) \nonumber\\
&  =\left[  n\right]  _{q}z\left(  1-z\right)  z^{n-1}+%
%TCIMACRO{\dsum _{k=1}^{n-1}}%
%BeginExpansion
{\displaystyle\sum_{k=1}^{n-1}}
%EndExpansion
\left(  \left[  k\right]  _{q}-\left[  n\right]  _{q}z\right)  p_{n,k}\left(
q;z\right)  I_{k,m}\nonumber\\
&  =\left[  n\right]  _{q}z^{n}+%
%TCIMACRO{\dsum _{k=1}^{n-1}}%
%BeginExpansion
{\displaystyle\sum_{k=1}^{n-1}}
%EndExpansion
\left[  k\right]  _{q}p_{n,k}\left(  q;z\right)  I_{k,m}-\left[  n\right]
_{q}z%
%TCIMACRO{\dsum _{k=1}^{n-1}}%
%BeginExpansion
{\displaystyle\sum_{k=1}^{n-1}}
%EndExpansion
p_{n,k}\left(  q;z\right)  I_{k,m}-\left[  n\right]  _{q}z^{n+1}\nonumber\\
&  =\left[  n\right]  _{q}z^{n}+%
%TCIMACRO{\dsum _{k=1}^{n-1}}%
%BeginExpansion
{\displaystyle\sum_{k=1}^{n-1}}
%EndExpansion
\left[  k\right]  _{q}p_{n,k}\left(  q;z\right)  I_{k,m}-z\left[  n\right]
_{q}U_{n,q}\left(  e_{m};z\right) \nonumber\\
&  =\left[  n\right]  _{q}z^{n}+q^{-m}%
%TCIMACRO{\dsum _{k=1}^{n-1}}%
%BeginExpansion
{\displaystyle\sum_{k=1}^{n-1}}
%EndExpansion
p_{n,k}\left(  q;z\right)  \left(  q^{m}\left[  k\right]  _{q}+\left[
m\right]  _{q}-\left[  m\right]  _{q}\right)  I_{k,m}-z\left[  n\right]
_{q}U_{n,q}\left(  e_{m};z\right) \nonumber\\
&  =\left[  n\right]  _{q}z^{n}+q^{-m}%
%TCIMACRO{\dsum _{k=1}^{n-1}}%
%BeginExpansion
{\displaystyle\sum_{k=1}^{n-1}}
%EndExpansion
p_{n,k}\left(  q;z\right)  \left(  q^{m}\left[  k\right]  _{q}+\left[
m\right]  _{q}-\left[  m\right]  _{q}\right)  I_{k,m}-z\left[  n\right]
_{q}U_{n,q}\left(  e_{m};z\right) \nonumber\\
&  =q^{-m}\left(  q^{m}\left[  n\right]  _{q}+\left[  m\right]  _{q}-\left[
m\right]  _{q}\right)  z^{n}+q^{-m}\left[  n+m\right]  _{q}%
%TCIMACRO{\dsum _{k=1}^{n-1}}%
%BeginExpansion
{\displaystyle\sum_{k=1}^{n-1}}
%EndExpansion
p_{n,k}\left(  q;z\right)  I_{k,m+1}\nonumber\\
&  -q^{-m}\left[  m\right]  _{q}%
%TCIMACRO{\dsum _{k=1}^{n-1}}%
%BeginExpansion
{\displaystyle\sum_{k=1}^{n-1}}
%EndExpansion
p_{n,k}\left(  q;z\right)  I_{k,m}-z\left[  n\right]  _{q}U_{n,q}\left(
e_{m};z\right) \nonumber\\
&  =q^{-m}\left[  n+m\right]  _{q}U_{n,q}\left(  e_{m+1};z\right)
-q^{-m}\left[  m\right]  _{q}U_{n,q}\left(  e_{m};z\right)  -z\left[
n\right]  _{q}U_{n,q}\left(  e_{m};z\right)  , \label{q1}%
\end{align}
which implies the recurrence in the statement.
\end{proof}

Let%
\[
\Theta_{n,m}\left(  q;z\right)  :=U_{n,q}\left(  e_{m};z\right)  -z^{m}%
-\frac{1}{\left[  n+1\right]  _{q}}\left(  q%
%TCIMACRO{\dsum \limits_{i=1}^{m-1}}%
%BeginExpansion
{\displaystyle\sum\limits_{i=1}^{m-1}}
%EndExpansion
\left[  i\right]  _{q}+%
%TCIMACRO{\dsum \limits_{i=1}^{m-1}}%
%BeginExpansion
{\displaystyle\sum\limits_{i=1}^{m-1}}
%EndExpansion
\left[  i\right]  _{q^{-1}}\right)  z^{m-1}\left(  1-z\right)  .
\]
Using the recurrence formula (\ref{rec1}) we prove two more recurrence formulas.

\begin{lemma}
\label{Lem:vv}For all $m,n\in\mathbb{N}$ and $z\in\mathbb{C}$ we have%
\begin{align}
U_{n,q}\left(  e_{m};z\right)  -z^{m}  &  =\frac{q^{m-1}z\left(  1-z\right)
}{\left[  n+m-1\right]  _{q}}D_{q}U_{n,q}\left(  e_{m-1};z\right) \nonumber\\
&  +\frac{q^{m-1}\left[  n\right]  z+\left[  m-1\right]  _{q}}{\left[
n+m-1\right]  _{q}}\left(  U_{n,q}\left(  e_{m-1};z\right)  -z^{m-1}\right)
+\frac{\left[  m-1\right]  _{q}}{\left[  n+m-1\right]  _{q}}\left(
1-z\right)  z^{m-1},\label{idd1}\\
\Theta_{n,m}\left(  q;z\right)   &  =\frac{q^{m-1}z\left(  1-z\right)
}{\left[  n+m-1\right]  _{q}}D_{q}\left(  U_{n,q}\left(  e_{m-1};z\right)
-z^{m-1}\right) \nonumber\\
&  +\frac{q^{m-1}\left[  n\right]  z+\left[  m-1\right]  _{q}}{\left[
n+m-1\right]  _{q}}\Theta_{n,m-1}\left(  q;z\right)  +R_{n,m}\left(
q;z\right)  , \label{idd2}%
\end{align}
where%
\begin{equation}
R_{n,m}\left(  q;z\right)  =\frac{\left[  m-1\right]  _{q}}{\left[
n+m-1\right]  _{q}\left[  n+1\right]  _{q}}\left[  \left(  1+q^{m-1}\right)
+\left(  q%
%TCIMACRO{\dsum \limits_{i=1}^{m-2}}%
%BeginExpansion
{\displaystyle\sum\limits_{i=1}^{m-2}}
%EndExpansion
\left[  i\right]  _{q}+%
%TCIMACRO{\dsum \limits_{i=1}^{m-2}}%
%BeginExpansion
{\displaystyle\sum\limits_{i=1}^{m-2}}
%EndExpansion
\left[  i\right]  _{q^{-1}}\right)  \left(  z+1\right)  \right]
z^{m-2}\left(  1-z\right)  . \label{rem1}%
\end{equation}

\end{lemma}

\begin{proof}
From the recurrence formula in Lemma \ref{Lem:rec}, for all $m\geq2$ we get%
\begin{align*}
U_{n,q}\left(  e_{m};z\right)  -z^{m}  &  =\frac{q^{m-1}z\left(  1-z\right)
}{\left[  n+m-1\right]  _{q}}D_{q}U_{n,q}\left(  e_{m-1};z\right)
+\frac{q^{m-1}\left[  n\right]  z+\left[  m-1\right]  _{q}}{\left[
n+m-1\right]  _{q}}\left(  U_{n,q}\left(  e_{m-1};z\right)  -z^{m-1}\right) \\
&  +\frac{q^{m-1}\left[  n\right]  z+\left[  m-1\right]  _{q}}{\left[
n+m-1\right]  _{q}}z^{m-1}-z^{m}\\
&  =\frac{q^{m-1}z\left(  1-z\right)  }{\left[  n+m-1\right]  _{q}}%
D_{q}U_{n,q}\left(  e_{m-1};z\right)  +\frac{q^{m-1}\left[  n\right]
z+\left[  m-1\right]  _{q}}{\left[  n+m-1\right]  _{q}}\left(  U_{n,q}\left(
e_{m-1};z\right)  -z^{m-1}\right) \\
&  +\frac{\left[  m-1\right]  _{q}}{\left[  n+m-1\right]  _{q}}\left(
1-z\right)  z^{m-1},
\end{align*}
and%
\begin{align*}
&  U_{n,q}\left(  e_{m};z\right)  -z^{m}-\frac{1}{\left[  n+1\right]  _{q}%
}\left(  q%
%TCIMACRO{\dsum \limits_{i=1}^{m-1}}%
%BeginExpansion
{\displaystyle\sum\limits_{i=1}^{m-1}}
%EndExpansion
\left[  i\right]  _{q}+%
%TCIMACRO{\dsum \limits_{i=1}^{m-1}}%
%BeginExpansion
{\displaystyle\sum\limits_{i=1}^{m-1}}
%EndExpansion
\left[  i\right]  _{q^{-1}}\right)  z^{m-1}\left(  1-z\right) \\
&  =\frac{q^{m-1}z\left(  1-z\right)  }{\left[  n+m-1\right]  _{q}}%
D_{q}\left(  U_{n,q}\left(  e_{m-1};z\right)  -z^{m-1}\right) \\
&  +\frac{q^{m-1}\left[  n\right]  z+\left[  m-1\right]  _{q}}{\left[
n+m-1\right]  _{q}}\left(  U_{n,q}\left(  e_{m};z\right)  -z^{m-1}-\frac
{1}{\left[  n+1\right]  _{q}}\left(  q%
%TCIMACRO{\dsum \limits_{i=1}^{m-2}}%
%BeginExpansion
{\displaystyle\sum\limits_{i=1}^{m-2}}
%EndExpansion
\left[  i\right]  _{q}+%
%TCIMACRO{\dsum \limits_{i=1}^{m-2}}%
%BeginExpansion
{\displaystyle\sum\limits_{i=1}^{m-2}}
%EndExpansion
\left[  i\right]  _{q^{-1}}\right)  z^{m-2}\left(  1-z\right)  \right) \\
&  +R_{n,m}\left(  q;z\right)  ,
\end{align*}
where%
\begin{align*}
R_{n,m}\left(  q;z\right)   &  =\frac{\left[  m-1\right]  _{q}}{\left[
n+m-1\right]  _{q}}\left(  1-z\right)  z^{m-1}-\frac{1}{\left[  n+1\right]
_{q}}\left(  q%
%TCIMACRO{\dsum \limits_{i=1}^{m-1}}%
%BeginExpansion
{\displaystyle\sum\limits_{i=1}^{m-1}}
%EndExpansion
\left[  i\right]  _{q}+%
%TCIMACRO{\dsum \limits_{i=1}^{m-1}}%
%BeginExpansion
{\displaystyle\sum\limits_{i=1}^{m-1}}
%EndExpansion
\left[  i\right]  _{q^{-1}}\right)  z^{m-1}\left(  1-z\right) \\
&  +\frac{q^{m-1}\left[  m-1\right]  _{q}}{\left[  n+m-1\right]  _{q}}\left(
1-z\right)  z^{m-1}\\
&  +\frac{q^{m-1}\left[  n\right]  z+\left[  m-1\right]  _{q}}{\left[
n+m-1\right]  _{q}}\frac{1}{\left[  n+1\right]  _{q}}\left(  q%
%TCIMACRO{\dsum \limits_{i=1}^{m-2}}%
%BeginExpansion
{\displaystyle\sum\limits_{i=1}^{m-2}}
%EndExpansion
\left[  i\right]  _{q}+%
%TCIMACRO{\dsum \limits_{i=1}^{m-2}}%
%BeginExpansion
{\displaystyle\sum\limits_{i=1}^{m-2}}
%EndExpansion
\left[  i\right]  _{q^{-1}}\right)  z^{m-2}\left(  1-z\right) \\
&  :=T_{n^{\prime}m}\left(  q\right)  z^{m-1}\left(  1-z\right)
+\frac{\left[  m-1\right]  _{q}}{\left[  n+m-1\right]  _{q}\left[  n+1\right]
_{q}}\left(  q%
%TCIMACRO{\dsum \limits_{i=1}^{m-2}}%
%BeginExpansion
{\displaystyle\sum\limits_{i=1}^{m-2}}
%EndExpansion
\left[  i\right]  _{q}+%
%TCIMACRO{\dsum \limits_{i=1}^{m-2}}%
%BeginExpansion
{\displaystyle\sum\limits_{i=1}^{m-2}}
%EndExpansion
\left[  i\right]  _{q^{-1}}\right)  z^{m-2}\left(  1-z\right)  .
\end{align*}
Again by simple calculation we obtain%
\begin{align*}
T_{n,m}\left(  q\right)   &  =\frac{\left[  m-1\right]  _{q}}{\left[
n+m-1\right]  _{q}}-\frac{1}{\left[  n+1\right]  _{q}}\left(  q%
%TCIMACRO{\dsum \limits_{i=1}^{m-1}}%
%BeginExpansion
{\displaystyle\sum\limits_{i=1}^{m-1}}
%EndExpansion
\left[  i\right]  _{q}+%
%TCIMACRO{\dsum \limits_{i=1}^{m-1}}%
%BeginExpansion
{\displaystyle\sum\limits_{i=1}^{m-1}}
%EndExpansion
\left[  i\right]  _{q^{-1}}\right) \\
&  +\frac{q^{m-1}\left[  m-1\right]  _{q}}{\left[  n+m-1\right]  _{q}}%
+\frac{q^{m-1}\left[  n\right]  _{q}}{\left[  n+m-1\right]  _{q}}\frac
{1}{\left[  n+1\right]  _{q}}\left(  q%
%TCIMACRO{\dsum \limits_{i=1}^{m-1}}%
%BeginExpansion
{\displaystyle\sum\limits_{i=1}^{m-1}}
%EndExpansion
\left[  i\right]  _{q}+%
%TCIMACRO{\dsum \limits_{i=1}^{m-1}}%
%BeginExpansion
{\displaystyle\sum\limits_{i=1}^{m-1}}
%EndExpansion
\left[  i\right]  _{q^{-1}}\right) \\
&  -\frac{q^{m-1}\left[  n\right]  _{q}}{\left[  n+m-1\right]  _{q}}\frac
{1}{\left[  n+1\right]  _{q}}\left(  q\left[  m-1\right]  _{q}+\left[
m-1\right]  _{q^{-1}}\right) \\
&  =\left(  \frac{\left[  m-1\right]  _{q}}{\left[  n+m-1\right]  _{q}}%
+\frac{q^{m-1}\left[  m-1\right]  _{q}}{\left[  n+m-1\right]  _{q}}%
-\frac{q^{m-1}\left[  n\right]  _{q}}{\left[  n+m-1\right]  _{q}}\frac
{1}{\left[  n+1\right]  _{q}}\left(  q\left[  m-1\right]  _{q}+\left[
m-1\right]  _{q^{-1}}\right)  \right) \\
&  +\left(  \frac{q^{m-1}\left[  n\right]  _{q}}{q^{m-1}\left[  n\right]
_{q}+\left[  m-1\right]  _{q}}-1\right)  \frac{1}{\left[  n+1\right]  _{q}%
}\left(  q%
%TCIMACRO{\dsum \limits_{i=1}^{m-1}}%
%BeginExpansion
{\displaystyle\sum\limits_{i=1}^{m-1}}
%EndExpansion
\left[  i\right]  _{q}+%
%TCIMACRO{\dsum \limits_{i=1}^{m-1}}%
%BeginExpansion
{\displaystyle\sum\limits_{i=1}^{m-1}}
%EndExpansion
\left[  i\right]  _{q^{-1}}\right) \\
&  :=T_{n,m}^{1}\left(  q\right)  +T_{n,m}^{2}\left(  q\right)  ,
\end{align*}
where $T_{n,m}^{1}\left(  q\right)  $ and $T_{n,m}^{2}\left(  q\right)  $ can
be simplified as follows:%
\begin{align*}
T_{n,m}^{2}\left(  q\right)   &  =\left(  1-\frac{q^{m-1}\left[  n\right]
_{q}}{\left[  n+m-1\right]  _{q}}\right)  \frac{1}{\left[  n+1\right]  _{q}%
}\left(  q%
%TCIMACRO{\dsum \limits_{i=1}^{m-2}}%
%BeginExpansion
{\displaystyle\sum\limits_{i=1}^{m-2}}
%EndExpansion
\left[  i\right]  _{q}+%
%TCIMACRO{\dsum \limits_{i=1}^{m-2}}%
%BeginExpansion
{\displaystyle\sum\limits_{i=1}^{m-2}}
%EndExpansion
\left[  i\right]  _{q^{-1}}\right) \\
&  =\frac{\left[  m-1\right]  _{q}}{\left[  n+m-1\right]  _{q}\left[
n+1\right]  _{q}}\left(  q%
%TCIMACRO{\dsum \limits_{i=1}^{m-2}}%
%BeginExpansion
{\displaystyle\sum\limits_{i=1}^{m-2}}
%EndExpansion
\left[  i\right]  _{q}+%
%TCIMACRO{\dsum \limits_{i=1}^{m-2}}%
%BeginExpansion
{\displaystyle\sum\limits_{i=1}^{m-2}}
%EndExpansion
\left[  i\right]  _{q^{-1}}\right)
\end{align*}
and%

\begin{align*}
&  T_{n,m}^{1}\left(  q\right)  =\frac{\left[  m-1\right]  _{q}}{\left[
n+m-1\right]  _{q}}+\frac{q^{m-1}\left[  m-1\right]  _{q}}{\left[
n+m-1\right]  _{q}}\\
&  -\frac{q^{m-1}\left[  n\right]  _{q}}{\left[  n+m-1\right]  _{q}}\frac
{1}{\left[  n+1\right]  _{q}}\left(  q\left[  m-1\right]  _{q}+\left[
m-1\right]  _{q^{-1}}\right) \\
&  =\left[  m-1\right]  _{q}\left(  \frac{1}{\left[  n+m-1\right]  _{q}}%
-\frac{q}{\left[  n+1\right]  _{q}}\frac{q^{m-1}\left[  n\right]  _{q}%
}{\left[  n+m-1\right]  _{q}}\right) \\
&  +\left[  m-1\right]  _{q}\left(  \frac{q^{m-1}}{\left[  n+m-1\right]  _{q}%
}-\frac{1}{\left[  n+1\right]  _{q}}\frac{q\left[  n\right]  _{q}}{\left[
n+m-1\right]  _{q}}\right) \\
&  =\left[  m-1\right]  _{q}\frac{\left[  n+1\right]  _{q}-q^{m}\left[
n\right]  _{q}}{\left[  n+m-1\right]  _{q}\left[  n+1\right]  _{q}}+\left[
m-1\right]  _{q}\frac{q^{m-1}\left[  n+1\right]  _{q}-q\left[  n\right]  _{q}%
}{\left[  n+m-1\right]  _{q}\left[  n+1\right]  _{q}}\\
&  =\left[  m-1\right]  _{q}\frac{\left(  1+q^{m-1}\right)  \left[
n+1\right]  _{q}-\left(  1+q^{m-1}\right)  q\left[  n\right]  _{q}}{\left[
n+m-1\right]  _{q}\left[  n+1\right]  _{q}}\\
&  =\frac{\left[  m-1\right]  _{q}\left(  1+q^{m-1}\right)  }{\left[
n+m-1\right]  _{q}\left[  n+1\right]  _{q}}.
\end{align*}

\end{proof}

\begin{lemma}
\label{Lem:lq}Let $q>1$ and $f\in H\left(  \mathbb{D}_{R}\right)  $. The
function $L_{q}\left(  f;z\right)  $ has the following representation
\[
L_{q}\left(  f;z\right)  =%
%TCIMACRO{\dsum \limits_{m=2}^{\infty}}%
%BeginExpansion
{\displaystyle\sum\limits_{m=2}^{\infty}}
%EndExpansion
a_{m}\left(  q%
%TCIMACRO{\dsum \limits_{i=1}^{m-1}}%
%BeginExpansion
{\displaystyle\sum\limits_{i=1}^{m-1}}
%EndExpansion
\left[  i\right]  +%
%TCIMACRO{\dsum \limits_{i=1}^{m-1}}%
%BeginExpansion
{\displaystyle\sum\limits_{i=1}^{m-1}}
%EndExpansion
\left[  i\right]  _{q^{-1}}\right)  z^{m-1}\left(  1-z\right)  ,\ \ \ z\in
\mathbb{D}_{R}.
\]

\end{lemma}

\begin{proof}
Using the following identity
\begin{align*}
\left[  m\right]  _{q}-m  &  =1+q+q^{2}+...q^{m-1}-m\\
&  =\left(  1-1\right)  +\left(  q-1\right)  +\left(  q^{2}-1\right)
+...+\left(  q^{m-1}-1\right) \\
&  =\left(  q-1\right)  \left[  1\right]  _{q}+\left(  q-1\right)  \left[
2\right]  _{q}+...+\left(  q-1\right)  \left[  m-1\right]  _{q}+\\
&  =\left(  q-1\right)  \left(  \left[  1\right]  _{q}+...+\left[  m-1\right]
_{q}\right)  =\left(  q-1\right)
%TCIMACRO{\dsum \limits_{i=1}^{m-1}}%
%BeginExpansion
{\displaystyle\sum\limits_{i=1}^{m-1}}
%EndExpansion
\left[  i\right]  _{q},
\end{align*}
we get
\begin{align*}
L_{q}\left(  f;z\right)   &  =\sum_{m=2}^{\infty}a_{m}\left(  \frac{q\left(
\left[  m\right]  _{q}-\left[  m\right]  _{q^{-1}}\right)  }{q-1}\right)
z^{m-1}\left(  1-z\right) \\
&  =\sum_{m=2}^{\infty}a_{m}\left(  \frac{q\left(  \left[  m\right]
_{q}-m\right)  }{q-1}+\frac{\left[  m\right]  _{q^{-1}}-m}{q^{-1}-1}\right)
z^{m-1}\left(  1-z\right) \\
&  =%
%TCIMACRO{\dsum \limits_{m=2}^{\infty}}%
%BeginExpansion
{\displaystyle\sum\limits_{m=2}^{\infty}}
%EndExpansion
a_{m}\left(  q%
%TCIMACRO{\dsum \limits_{i=1}^{m-1}}%
%BeginExpansion
{\displaystyle\sum\limits_{i=1}^{m-1}}
%EndExpansion
\left[  i\right]  _{q}+%
%TCIMACRO{\dsum \limits_{i=1}^{m-1}}%
%BeginExpansion
{\displaystyle\sum\limits_{i=1}^{m-1}}
%EndExpansion
\left[  i\right]  _{q^{-1}}\right)  z^{m-1}\left(  1-z\right)  ,
\end{align*}
where $f\left(  z\right)  =\sum_{m=0}^{\infty}a_{m}z^{m}$.
\end{proof}

\section{Proofs of the main results}

Firstly we prove that $U_{n,q}(f;z)=\sum_{m=0}^{\infty}a_{m}U_{n,q}(e_{m},z)$.
Indeed denoting $f_{k}(z)=\sum_{j=0}^{k}a_{j}z^{j},|z|\leq r$ with
$m\in\mathbb{N}$, by the linearity of $U_{n,q}$, we have
\[
U_{n,q}(f_{k},z)=\sum_{m=0}^{k}a_{m}U_{n,q}(e_{m},z),
\]
and it is sufficient to show that for any fixed $n\in\mathbb{N}$ and $|z|\leq
r$ with $r\geq1$, we have $\lim_{k\rightarrow\infty}U_{n,q}(f_{k}%
,z)=U_{n,q}(f;z)$. But this is immediate from $\lim_{k\rightarrow\infty
}||f_{k}-f||_{r}=0$, the norm being the defined as $||f||_{r}=\max
\{|f(z)|:{|z|\leq r}\}$ and from the inequality
\begin{align*}
&  |U_{n,q}(f_{k},z)-U_{n,q}(f,z)|\\
&  \leq|f_{k}(0)-f(0)|\cdot|(1-z)^{n}|+|f_{k}(1)-f(1)|\cdot|z^{n}|\\
&  +[n+1]_{q^{-1}}\sum_{j=1}^{n-1}|p_{n,j}(q;z)|q^{j-1}\int_{0}^{1}%
p_{n-2,j-1}(q^{-1},q^{-1}t)|f_{k}(t)-f(t)|d_{q^{-1}}t\\
&  \leq C_{r,n}||f_{k}-f||_{r},
\end{align*}
valid for all $|z|\leq r$, where
\begin{align*}
C_{r,n}  &  =(1+r)^{n}+r^{n}+[n+1]_{q^{-1}}\sum_{j=1}^{n-1}\left[
\begin{array}
[c]{c}%
n\\
j
\end{array}
\right]  _{q}(1+r)^{n-j}r^{j}q^{j-1}\int_{0}^{1}p_{n-2,j-1}(q^{-1}%
;q^{-1}t)d_{q^{-1}}t\\
&  =(1+r)^{n}+r^{n}+\sum_{j=1}^{n-1}\left[
\begin{array}
[c]{c}%
n\\
j
\end{array}
\right]  _{q}(1+r)^{n-j}r^{j}q^{j-1}.
\end{align*}
Therefore we get
\[
|U_{n,q}(f;z)-f(z)|\leq\sum_{m=0}^{\infty}|a_{m}||U_{n,q}(e_{m},z)-e_{m}%
(z)|=\sum_{m=2}^{\infty}|a_{m}||U_{n,q}(e_{m},z)-e_{m}(z)|,
\]
as $U_{n,q}(e_{0},z)=e_{0}(z)$ and .$U_{n,q}(e_{1},z)=e_{1}(z).$

\begin{proof}
[Proof of Theorem \ref{t:convergence}]From the recurrence formula (\ref{idd1})
and the inequality (\ref{inq1}) for $m\geq2$ we get%
\begin{align*}
\left\vert U_{n,q}\left(  e_{m};z\right)  -z^{m}\right\vert  &  \leq
\frac{q^{m-1}z\left(  1-z\right)  }{q^{m-2}\left[  n+1\right]  _{q}+\left[
m-2\right]  _{q}}\left\vert D_{q}U_{n,q}\left(  e_{m-1};z\right)  \right\vert
\\
&  +\frac{q^{m-1}\left[  n\right]  z+\left[  m-1\right]  _{q}}{q^{m-1}\left[
n\right]  _{q}+\left[  m-1\right]  _{q}}\left\vert U_{n,q}\left(
e_{m-1};z\right)  -z^{m-1}\right\vert +\frac{\left[  m-1\right]  _{q}}%
{q^{m-2}\left[  n+1\right]  _{q}+\left[  m-2\right]  _{q}}\left\vert
1-z\right\vert \left\vert z\right\vert ^{m-1}.
\end{align*}
It is known that by a linear transformation, the Bernstein inequality in the
closed unit disk becomes%
\[
\left\vert P_{k}^{\prime}\left(  z\right)  \right\vert \leq\frac{k}{qr_{1}%
}\left\Vert P_{k}\right\Vert _{qr},\ \ \ \text{for all\ \ }\left\vert
z\right\vert \leq qr,\ \ r\geq1,
\]
which combined with the mean value theorem in complex analysis implies%
\[
\left\vert D_{q}\left(  P_{k};z\right)  \right\vert \leq\left\Vert
P_{k}^{\prime}\right\Vert _{qr}\leq\frac{k}{qr}\left\Vert P_{k}\right\Vert
_{qr},
\]
for all $\left\vert z\right\vert \leq qr$, where $P_{k}\left(  z\right)  $ is
a complex polynomial of degree $\leq k$. It follows that%
\begin{align*}
&  \left\vert U_{n,q}\left(  e_{m};z\right)  -z^{m}\right\vert \\
&  \leq\frac{q^{m-1}r\left(  1+r\right)  }{q^{m-2}\left[  n+1\right]
_{q}+\left[  m-2\right]  _{q}}\frac{m-1}{qr}\left\Vert U_{n,q}\left(
e_{m-1}\right)  \right\Vert _{qr}\\
&  +r\left\vert U_{n,q}\left(  e_{m-1};z\right)  -z^{m-1}\right\vert
+\frac{\left[  m-1\right]  _{1/q}}{\left[  n+1\right]  _{q}}\left(
1+r\right)  r^{m-1}\\
&  \leq\frac{\left(  m-1\right)  }{\left[  n+1\right]  _{q}}\left(
1+r\right)  q^{m-1}r^{m-1}+r\left\vert U_{n,q}\left(  e_{m-1};z\right)
-z^{m-1}\right\vert +\frac{\left[  m-1\right]  _{1/q}}{\left[  n+1\right]
_{q}}\left(  1+r\right)  r^{m-1}\\
&  \leq2q\left(  m-1\right)  \frac{r\left(  1+r\right)  }{\left[  n+1\right]
_{q}}\left(  qr\right)  ^{m-2}+r\left\vert U_{n,q}\left(  e_{m-1};z\right)
-z^{m-1}\right\vert .
\end{align*}
By writing the last inequality for $m=2,3,...,$ we easily obtain, step by step
the following%
\begin{align*}
\left\vert U_{n,q}\left(  e_{m};z\right)  -z^{m}\right\vert  &  \leq r\left(
r\left\vert U_{n,q}\left(  e_{m-2};z\right)  -z^{m-2}\right\vert
+2\frac{\left(  m-2\right)  }{\left[  n+1\right]  _{q}}r\left(  1+r\right)
\left(  qr\right)  ^{m-3}\right) \\
&  +2\frac{\left(  m-1\right)  }{\left[  n+1\right]  _{q}}r\left(  1+r\right)
\left(  qr\right)  ^{m-2}\\
&  =r^{2}\left\vert U_{n,q}\left(  e_{m-2};z\right)  -z^{m-2}\right\vert
+2\frac{r\left(  1+r\right)  }{\left[  n+1\right]  _{q}}r^{m-2}\left(
m-1+m-2\right) \\
&  \leq...\leq\frac{r\left(  1+r\right)  }{\left[  n+1\right]  _{q}}m\left(
m-1\right)  q^{m-2}r^{m-2}.
\end{align*}
It follows that%
\[
\left\vert U_{n,q}\left(  f;z\right)  -f\left(  z\right)  \right\vert \leq%
%TCIMACRO{\dsum \limits_{m=2}^{\infty}}%
%BeginExpansion
{\displaystyle\sum\limits_{m=2}^{\infty}}
%EndExpansion
\left\vert a_{m}\right\vert \left\vert U_{n,q}\left(  e_{m};z\right)
-z^{m}\right\vert \leq\frac{r\left(  1+r\right)  }{\left[  n+1\right]  _{q}}%
%TCIMACRO{\dsum \limits_{m=2}^{\infty}}%
%BeginExpansion
{\displaystyle\sum\limits_{m=2}^{\infty}}
%EndExpansion
\left\vert a_{m}\right\vert m\left(  m-1\right)  q^{m-2}r^{m-2}.
\]

\end{proof}

The second main result of the paper is the Voronovskaja theorem with a
quantitative estimate for the complex version of genuine $q$%
-Bernstein-Durrmeyer polynomials.

\begin{proof}
[Proof of Theorem \ref{t:voronovskaja}]By Lemma \ref{Lem:vv} we have%
\begin{align}
\Theta_{n,m}\left(  q;z\right)   &  =\frac{q^{m-1}z\left(  1-z\right)
}{\left[  n+m-1\right]  _{q}}D_{q}\left(  U_{n,q}\left(  e_{m-1};z\right)
-z^{m-1}\right) \label{rr2}\\
&  +\frac{q^{m-1}\left[  n\right]  z+\left[  m-1\right]  _{q}}{\left[
n+m-1\right]  _{q}}\Theta_{n,m-1}\left(  q;z\right)  +R_{n,m}\left(
q;z\right)  ,\nonumber
\end{align}
where%
\[
R_{n,m}\left(  q;z\right)  =\frac{\left[  m-1\right]  _{q}}{\left[
n+m-1\right]  _{q}\left[  n+1\right]  _{q}}\left[  \left(  1+q^{m-1}\right)
+\left(  q%
%TCIMACRO{\dsum \limits_{i=1}^{m-2}}%
%BeginExpansion
{\displaystyle\sum\limits_{i=1}^{m-2}}
%EndExpansion
\left[  i\right]  _{q}+%
%TCIMACRO{\dsum \limits_{i=1}^{m-2}}%
%BeginExpansion
{\displaystyle\sum\limits_{i=1}^{m-2}}
%EndExpansion
\left[  i\right]  _{q^{-1}}\right)  \left(  z+1\right)  \right]
z^{m-2}\left(  1-z\right)  .
\]
It follows that%
\begin{align*}
\left\vert R_{n,m}\left(  q;z\right)  \right\vert  &  \leq\frac{\left[
m-1\right]  _{q}}{\left[  n+1\right]  _{q}^{2}}\left(  \left(  1+q^{m-1}%
\right)  r+\left(  q%
%TCIMACRO{\dsum \limits_{i=1}^{m-2}}%
%BeginExpansion
{\displaystyle\sum\limits_{i=1}^{m-2}}
%EndExpansion
\left[  i\right]  _{q}+%
%TCIMACRO{\dsum \limits_{i=1}^{m-2}}%
%BeginExpansion
{\displaystyle\sum\limits_{i=1}^{m-2}}
%EndExpansion
\left[  i\right]  _{q^{-1}}\right)  \left(  1+r\right)  \right)  \left(
1+r\right)  r^{m-2}\\
&  \leq\frac{\left[  m-1\right]  _{q}}{\left[  n+1\right]  _{q}^{2}}\left(
\left(  1+q^{m-1}\right)  +\left(  q\left(  m-2\right)  \left[  m-2\right]
_{q}+\left(  m-2\right)  ^{2}\right)  \right)  \left(  1+r\right)  ^{2}%
r^{m-2}\\
&  =\frac{q^{m-2}\left[  m-1\right]  _{q^{-1}}}{\left[  n+1\right]  _{q}^{2}%
}q^{m-2}\left(  \left(  \frac{1}{q^{m-2}}+q\right)  +\left(  m-2\right)
\left[  m-2\right]  _{q^{-1}}+\frac{1}{q^{m-2}}\left(  m-2\right)
^{2}\right)  \left(  1+r\right)  ^{2}r^{m-2}\\
&  \leq\frac{3}{\left[  n+1\right]  _{q}^{2}}\left(  m-1\right)  \left(
m-2\right)  ^{2}\left(  1+r\right)  ^{2}\left(  q^{2}r\right)  ^{m-2}%
\end{align*}
for all $m\geq2$, $n\in\mathbb{N}$ and $z\in\mathbb{C}$. (\ref{rr2}) implies
that for $\left\vert z\right\vert \leq r$
\begin{align*}
\left\vert \Theta_{n,m}\left(  q;z\right)  \right\vert  &  \leq r\left\vert
\Theta_{n,m-1}\left(  q;z\right)  \right\vert +\frac{q^{m-1}r\left(
1+r\right)  }{q^{m-2}\left[  n+1\right]  _{q}}\frac{m-1}{qr}\left\Vert
U_{n,q}\left(  e_{m-1}\right)  -e_{m-1}\right\Vert _{qr}\\
&  +\frac{3}{\left[  n+1\right]  _{q}^{2}}\left(  m-1\right)  \left(
m-2\right)  ^{2}\left(  1+r\right)  ^{2}\left(  q^{2}r\right)  ^{m-2}\\
&  \leq r\left\vert \Theta_{n,m-1}\left(  q;z\right)  \right\vert +\frac
{r^{2}\left(  1+r\right)  ^{2}}{\left[  n+1\right]  _{q}^{2}}\left(
m-1\right)  ^{2}\left(  m-2\right)  \left(  q^{2}r\right)  ^{m-3}\\
&  +\frac{3}{\left[  n+1\right]  _{q}^{2}}\left(  m-1\right)  \left(
m-2\right)  ^{2}\left(  1+r\right)  ^{2}\left(  q^{2}r\right)  ^{m-2}\\
&  \leq r\left\vert \Theta_{n,m-1}\left(  q;z\right)  \right\vert
+\frac{4r^{2}\left(  1+r\right)  ^{2}}{\left[  n+1\right]  _{q}^{2}}\left(
m-1\right)  ^{2}\left(  m-2\right)  \left(  q^{2}r\right)  ^{m-2}.
\end{align*}
By writing the last inequality for $m=3,4,...,$ we easily obtain, step by step
the following%
\begin{align*}
&  \left\vert U_{n,q}\left(  f;z\right)  -f\left(  z\right)  -\frac{1}{\left[
n+1\right]  _{q}}L_{q}\left(  f;z\right)  \right\vert \\
&  \leq\frac{4r^{2}\left(  1+r\right)  ^{2}}{\left[  n+1\right]  _{q}^{2}}%
%TCIMACRO{\dsum \limits_{m=2}^{\infty}}%
%BeginExpansion
{\displaystyle\sum\limits_{m=2}^{\infty}}
%EndExpansion
\left\vert a_{m}\right\vert \left(  q^{2}r\right)  ^{m-2}%
%TCIMACRO{\dsum \limits_{j=2}^{m}}%
%BeginExpansion
{\displaystyle\sum\limits_{j=2}^{m}}
%EndExpansion
\left(  j-1\right)  ^{2}\left(  j-2\right)  \leq\frac{4r^{2}\left(
1+r\right)  ^{2}}{\left[  n+1\right]  _{q}^{2}}%
%TCIMACRO{\dsum \limits_{m=2}^{\infty}}%
%BeginExpansion
{\displaystyle\sum\limits_{m=2}^{\infty}}
%EndExpansion
\left\vert a_{m}\right\vert \left(  m-1\right)  ^{2}\left(  m-2\right)
^{2}\left(  q^{2}r\right)  ^{m-2}.
\end{align*}

\end{proof}

\begin{proof}
[Proof of Theorem \ref{t:exactdegree}]For all $z\in\mathbb{D}_{R}$ and
$n\in\mathbb{N}$ we get%
\[
U_{n,q}\left(  f;z\right)  -f\left(  z\right)  =\frac{1}{\left[  n+1\right]
_{q}}\left\{  L_{q}\left(  f;z\right)  +\left[  n+1\right]  _{q}\left(
U_{n,q}\left(  f;z\right)  -f\left(  z\right)  -\frac{1}{\left[  n+1\right]
_{q}}L_{q}\left(  f;z\right)  \right)  \right\}  .
\]
It follows that%
\[
\left\Vert U_{n,q}\left(  f\right)  -f\right\Vert _{r}\geq\frac{1}{\left[
n+1\right]  _{q}}\left\{  \left\Vert L_{q}\left(  f;z\right)  \right\Vert
_{r}-\left[  n+1\right]  _{q}\left\Vert U_{n,q}\left(  f\right)  -f-\frac
{1}{\left[  n+1\right]  _{q}}L_{q}\left(  f;z\right)  \right\Vert
_{r}\right\}  .
\]
Because by hypothesis $f$ is not a polynomial of degree $\leq1$ in
$\mathbb{D}_{R}$, it follows $\left\Vert L_{q}\left(  f;z\right)  \right\Vert
_{r}>0$. Indeed, assuming the contrary it follows that $L_{q}\left(
f;z\right)  =0$ for all $z\in\mathbb{D}_{r}$ that is $D_{q}f\left(  z\right)
=D_{q^{-1}}f\left(  z\right)  $ for all $z\in\mathbb{D}_{r}.$ Thus $a_{m}=0,$
$m=2,3,...$ and, $f$ is linear, which is contradiction with the hypothesis.

Now, by Theorem \ref{t:voronovskaja} we have%
\begin{align*}
\left[  n+1\right]  _{q}\left\vert U_{n,q}\left(  f;z\right)  -f\left(
z\right)  -\frac{1}{\left[  n+1\right]  _{q}}L_{q}\left(  f;z\right)
\right\vert  &  \leq\frac{4r^{2}\left(  1+r\right)  ^{2}}{\left[  n+1\right]
_{q}}%
%TCIMACRO{\dsum \limits_{m=3}^{\infty}}%
%BeginExpansion
{\displaystyle\sum\limits_{m=3}^{\infty}}
%EndExpansion
\left\vert a_{m}\right\vert \left(  m-1\right)  ^{2}\left(  m-2\right)
^{2}\left(  q^{2}r\right)  ^{m-2}\\
&  \rightarrow0\ \ \ \text{as\ \ \ }n\rightarrow\infty.
\end{align*}
Consequently, there exists $n_{1}$ (depending only on $f$ and $r$) such that
for all $n\geq n_{1}$ we have%
\[
\left\Vert L_{q}\left(  f;z\right)  \right\Vert _{r}-\left[  n+1\right]
_{q}\left\Vert U_{n,q}\left(  f\right)  -f-\frac{1}{\left[  n+1\right]  _{q}%
}L_{q}\left(  f;z\right)  \right\Vert _{r}\geq\frac{1}{2}\left\Vert
L_{q}\left(  f;z\right)  \right\Vert _{r},
\]
which implies%
\[
\left\Vert U_{n,q}\left(  f\right)  -f\right\Vert _{r}\geq\frac{1}{2\left[
n+1\right]  _{q}}\left\Vert L_{q}\left(  f;z\right)  \right\Vert
_{r},\ \ \ \text{for all }n\geq n_{1}.
\]
For $1\leq n\leq n_{1}-1$ we have
\[
\left\Vert U_{n,q}\left(  f\right)  -f\right\Vert _{r}\geq\frac{1}{\left[
n+1\right]  _{q}}\left(  \left[  n+1\right]  _{q}\left\Vert U_{n,q}\left(
f\right)  -f\right\Vert _{r}\right)  =\frac{1}{\left[  n+1\right]  _{q}%
}M_{r,n,q}\left(  f\right)  >0,
\]
which finally implies that
\[
\left\Vert U_{n,q}\left(  f\right)  -f\right\Vert _{r}\geq\frac{1}{\left[
n+1\right]  _{q}}C_{r,q}\left(  f\right)  ,
\]
for all $n$, with $C_{r,q}\left(  f\right)  =\min\left\{  M_{r,1,q}\left(
f\right)  ,...,M_{r,n_{1}-1,q}\left(  f\right)  ,\frac{1}{2}\left\Vert
L_{q}\left(  f;z\right)  \right\Vert _{r}\right\}  $, which ends the proof.
\end{proof}

\begin{proof}
[Proof of Theorem \ref{t:continuous}]Let $1\leq r<R$, let $1<q_{0}<\dfrac
{R}{r}$ be fixed. Then by Lemma \ref{Lem:lq} for any $1\leq q\leq q_{0}$ and
$\left\vert z\right\vert \leq r$, we have%
\begin{align*}
L_{q}\left(  f;z\right)   &  =%
%TCIMACRO{\dsum \limits_{m=2}^{\infty}}%
%BeginExpansion
{\displaystyle\sum\limits_{m=2}^{\infty}}
%EndExpansion
a_{m}\left(  q%
%TCIMACRO{\dsum \limits_{i=1}^{m-1}}%
%BeginExpansion
{\displaystyle\sum\limits_{i=1}^{m-1}}
%EndExpansion
\left[  i\right]  _{q}+%
%TCIMACRO{\dsum \limits_{i=1}^{m-1}}%
%BeginExpansion
{\displaystyle\sum\limits_{i=1}^{m-1}}
%EndExpansion
\left[  i\right]  _{q^{-1}}\right)  z^{m-1}\left(  1-z\right)  ,\\
L_{1}\left(  f;z\right)   &  =%
%TCIMACRO{\dsum \limits_{m=2}^{\infty}}%
%BeginExpansion
{\displaystyle\sum\limits_{m=2}^{\infty}}
%EndExpansion
a_{m}m\left(  m-1\right)  z^{m-1}\left(  1-z\right)  .
\end{align*}
Using the inequality%
\begin{align*}
\left\vert q%
%TCIMACRO{\dsum \limits_{i=1}^{m-1}}%
%BeginExpansion
{\displaystyle\sum\limits_{i=1}^{m-1}}
%EndExpansion
\left[  i\right]  _{q}-\dfrac{m\left(  m-1\right)  }{2}\right\vert  &  =q%
%TCIMACRO{\dsum \limits_{i=2}^{m-1}}%
%BeginExpansion
{\displaystyle\sum\limits_{i=2}^{m-1}}
%EndExpansion
\left(  \left[  i\right]  _{q}-i\right)  +\left(  q-1\right)  \dfrac{m\left(
m-1\right)  }{2}\\
&  =q\left(  q-1\right)
%TCIMACRO{\dsum \limits_{i=2}^{m-1}}%
%BeginExpansion
{\displaystyle\sum\limits_{i=2}^{m-1}}
%EndExpansion%
%TCIMACRO{\dsum \limits_{j=1}^{i}}%
%BeginExpansion
{\displaystyle\sum\limits_{j=1}^{i}}
%EndExpansion
\left[  j\right]  _{q}+\left(  q-1\right)  \dfrac{m\left(  m-1\right)  }{2}\\
&  \leq q\left(  q-1\right)  \left[  m-1\right]  _{q}\dfrac{m\left(
m-1\right)  }{2}+\left(  q-1\right)  \dfrac{m\left(  m-1\right)  }{2}\\
&  =\left(  q-1\right)  \dfrac{m\left(  m-1\right)  }{2}\left(  q\left[
m-1\right]  _{q}+1\right) \\
&  \leq\left(  q-1\right)  q^{m-1}\dfrac{m^{2}\left(  m-1\right)  }{2}%
\end{align*}
and%
\begin{align*}
\left\vert
%TCIMACRO{\dsum \limits_{i=1}^{m-1}}%
%BeginExpansion
{\displaystyle\sum\limits_{i=1}^{m-1}}
%EndExpansion
\left[  i\right]  _{q^{-1}}-\dfrac{m\left(  m-1\right)  }{2}\right\vert  &  =%
%TCIMACRO{\dsum \limits_{i=2}^{m-1}}%
%BeginExpansion
{\displaystyle\sum\limits_{i=2}^{m-1}}
%EndExpansion
\left(  i-\left[  i\right]  _{q^{-1}}\right) \\
&  =\left(  1-q^{-1}\right)
%TCIMACRO{\dsum \limits_{i=2}^{m-1}}%
%BeginExpansion
{\displaystyle\sum\limits_{i=2}^{m-1}}
%EndExpansion%
%TCIMACRO{\dsum \limits_{j=1}^{i}}%
%BeginExpansion
{\displaystyle\sum\limits_{j=1}^{i}}
%EndExpansion
\left[  j\right]  _{q^{-1}}\\
&  \leq\left(  1-q^{-1}\right)  \dfrac{m\left(  m-1\right)  ^{2}}{2},
\end{align*}
we get for $1\leq q\leq q_{0}$ and $\left\vert z\right\vert \leq r$,%
\begin{align*}
&  \left\vert L_{q}\left(  f;z\right)  -L_{1}\left(  f;z\right)  \right\vert
\\
&  \leq%
%TCIMACRO{\dsum \limits_{m=2}^{N-1}}%
%BeginExpansion
{\displaystyle\sum\limits_{m=2}^{N-1}}
%EndExpansion
\left\vert a_{m}\right\vert \left\vert q%
%TCIMACRO{\dsum \limits_{i=1}^{m-1}}%
%BeginExpansion
{\displaystyle\sum\limits_{i=1}^{m-1}}
%EndExpansion
\left[  i\right]  _{q}-\dfrac{m\left(  m-1\right)  }{2}\right\vert \left\vert
z^{m-1}-z^{m}\right\vert +%
%TCIMACRO{\dsum \limits_{m=N}^{\infty}}%
%BeginExpansion
{\displaystyle\sum\limits_{m=N}^{\infty}}
%EndExpansion
\left\vert a_{m}\right\vert \left\vert q%
%TCIMACRO{\dsum \limits_{i=1}^{m-1}}%
%BeginExpansion
{\displaystyle\sum\limits_{i=1}^{m-1}}
%EndExpansion
\left[  i\right]  _{q}-\dfrac{m\left(  m-1\right)  }{2}\right\vert \left\vert
z^{m-1}-z^{m}\right\vert \\
&  +%
%TCIMACRO{\dsum \limits_{m=2}^{N-1}}%
%BeginExpansion
{\displaystyle\sum\limits_{m=2}^{N-1}}
%EndExpansion
\left\vert a_{m}\right\vert \left\vert
%TCIMACRO{\dsum \limits_{i=1}^{m-1}}%
%BeginExpansion
{\displaystyle\sum\limits_{i=1}^{m-1}}
%EndExpansion
\left[  i\right]  _{q^{-1}}-\dfrac{m\left(  m-1\right)  }{2}\right\vert
\left\vert z^{m-1}-z^{m}\right\vert +%
%TCIMACRO{\dsum \limits_{m=N}^{\infty}}%
%BeginExpansion
{\displaystyle\sum\limits_{m=N}^{\infty}}
%EndExpansion
\left\vert a_{m}\right\vert \left\vert
%TCIMACRO{\dsum \limits_{i=1}^{m-1}}%
%BeginExpansion
{\displaystyle\sum\limits_{i=1}^{m-1}}
%EndExpansion
\left[  i\right]  _{q^{-1}}-\dfrac{m\left(  m-1\right)  }{2}\right\vert
\left\vert z^{m-1}-z^{m}\right\vert \\
&  \leq\left(  q-1\right)
%TCIMACRO{\dsum \limits_{m=2}^{N-1}}%
%BeginExpansion
{\displaystyle\sum\limits_{m=2}^{N-1}}
%EndExpansion
\left\vert a_{m}\right\vert m^{2}\left(  m-1\right)  q_{0}^{m-1}r^{m}+4%
%TCIMACRO{\dsum \limits_{m=N}^{\infty}}%
%BeginExpansion
{\displaystyle\sum\limits_{m=N}^{\infty}}
%EndExpansion
\left\vert a_{m}\right\vert \left(  m-1\right)  ^{2}q_{0}^{m}r^{m}\\
&  +\left(  1-q^{-1}\right)
%TCIMACRO{\dsum \limits_{m=2}^{N-1}}%
%BeginExpansion
{\displaystyle\sum\limits_{m=2}^{N-1}}
%EndExpansion
\left\vert a_{m}\right\vert m\left(  m-1\right)  ^{2}r^{m}+2%
%TCIMACRO{\dsum \limits_{m=N}^{\infty}}%
%BeginExpansion
{\displaystyle\sum\limits_{m=N}^{\infty}}
%EndExpansion
\left\vert a_{m}\right\vert m\left(  m-1\right)  r^{m}.
\end{align*}
Since $f\in H\left(  \mathbb{D}_{R}\right)  ,$ we can find $N=N_{\varepsilon
}\in\mathbb{N}$ such that%
\[
4%
%TCIMACRO{\dsum \limits_{m=N}^{\infty}}%
%BeginExpansion
{\displaystyle\sum\limits_{m=N}^{\infty}}
%EndExpansion
\left\vert a_{m}\right\vert \left(  m-1\right)  ^{2}q_{0}^{m}r^{m}+2%
%TCIMACRO{\dsum \limits_{m=N}^{\infty}}%
%BeginExpansion
{\displaystyle\sum\limits_{m=N}^{\infty}}
%EndExpansion
\left\vert a_{m}\right\vert m\left(  m-1\right)  r^{m}<\varepsilon/2.
\]
Thus for $q$ sufficiently close to $1$ from the right, we conclude that%
\[
\lim\limits_{q\rightarrow1^{+1}}L_{q}\left(  f;z\right)  =L_{1}\left(
f;z\right)
\]
uniformly on $\mathbb{D}_{r}$. The proof is finished.
\end{proof}

\begin{proof}
[Proof of Theorem \ref{t:saturation}]Then by Theorem \ref{t:voronovskaja}, we
get $L_{q}\left(  f;z\right)  =\lim_{n\rightarrow\infty}\left[  n+1\right]
_{q}\left(  U_{n,q}\left(  f;z\right)  -f\left(  z\right)  \right)  =0$ for
infinite number of points having an accumulation point on $\mathbb{D}_{R/q}$.
Since $L_{q}\left(  f;z\right)  \in H\left(  \mathbb{D}_{R/q}\right)  $, by
the Unicity Theorem for analytic functions we get $L_{q}\left(  f;z\right)
=0$ in $\mathbb{D}_{R/q}$, and therefore, by (\ref{lq}), $a_{m}=0$,
$m=2,3,...$ Thus, $f$ is linear. Theorem \ref{t:saturation} is proved.
\end{proof}

\end{document}